\documentclass[10pt, a4paper]{article}

\usepackage[round, authoryear]{natbib}
\usepackage{amssymb,amsmath,amsthm,amsfonts, mathtools, bbm}
\usepackage{caption, subcaption, graphbox}
\usepackage{tikz}
\usepackage{comment}
\newtheorem{theorem}{Theorem}[section]
\newtheorem{lemma}[theorem]{Lemma}

\newcommand{\dx}{\mathrm{d}}
\usepackage{xargs}
\usepackage{indentfirst}
\usepackage{txfonts}
\usepackage{authblk}

\usepackage[twoside,
  paperwidth=210mm,
  paperheight=297mm,
  textheight=622pt,
  textwidth=468pt,
  centering,
  headheight=50pt,
  headsep=12pt,
  footskip=18pt,
  footnotesep=24pt plus 2pt minus 12pt,
  columnsep=2pc]{geometry}
\usepackage{hyperref}
\begin{document}



\title{Exact results for the order picking time distribution under return routing}





\author[1]{Tim Engels}
\author[2]{Ivo Adan}
\author[1]{Onno Boxma}
\author[1]{Jacques Resing}
\affil[1]{Department of Mathematics and Computer Science, Eindhoven University of Technology}

\affil[2]{Department of Industrial Engineering and Innovation Sciences, Eindhoven University of Technology}
\maketitle

\begin{abstract}
This paper derives exact expressions for the Laplace-Stieltjes transform of the order picking time in single- and 2-block warehouses. We consider manual warehouses that deploy return routing and assume that order sizes follow a Poisson distribution. The results in this paper apply to a wide range of storage policies, including but not restricted to class-based and random storage. Furthermore, we compare the performance of the storage policies and warehouse lay-outs by using numerical inversion of the Laplace-Stieltjes transforms.
\end{abstract}




\section{Introduction}
\label{sec:Introduction}
Warehouses have grown to become an essential part of today's supply chains, as is illustrated in \citet{Boysen2019}. With this growth, the analysis and optimization of warehouse design has also become ever more important. In particular, the order picking process, the retrieval of items from a warehouse, stands central in many designs due to its high cost. \citet{Tompkins2011}, for example, estimated that order picking contributes at least $50\%$ of all operating costs.

The analysis and optimization of order picking in itself already is a challenging subject, which is highlighted by the vast amount of literature on the topic, see e.g. the literature review of \citet{Boysen2019}. This complexity stems from the many elements that affect the picking process, such as: the lay-out of warehouses, routing of pickers and storage assignment. On top of that, optimal solutions often are impractical, e.g., optimal routes cause confusion amongst pickers \citep{Gademann2001}. To overcome such practical issues, warehouses often apply straightforward, yet efficient, policies.

The performance evaluation of such policies is also crucial to warehousing design. Literature on this topic varies from simulation studies \citep{Petersen2004} to average route-length approximations \citep{Hall1993}. \citet{Dijkstra2017} derived exact formulas for the average route-length, which were then used to construct optimal class-based storage lay-outs. A different point of view came from \citet{Chew1999}, where the authors proposed to compare order batching policies based on the average order-lead time, i.e. the delay experienced by a customer, rather than the average order picking time. The authors modeled the warehouse as a multi-server queue and used a (two moment) approximation of the expected order-lead time to find the optimal batch size. More recently, \citet{Engels2022} provided exact results for the first two moments of the order picking time for several routing policies. Using similar queueing approximations the authors showed that the policy with the shortest average route-length, but higher second moment, could possibly result in higher average order-lead times.

In this paper we step away from the moments of the order picking time and instead derive exact results for the Laplace-Stieltjes transform (abv. LST) of the order picking time distribution. For this, we analyze both single- and 2-block warehouses that deploy the return routing policy. We assume that the order size follows a Poisson distribution and exploit that the numbers of picks in the different aisles are independent under this assumption. The results in this paper allow for the analysis of other properties of the picking time distribution, such as tail probabilities, and an accurate approximation of the order-lead time distribution \citep[see e.g.][pages 292-300]{Tijms}.

Multi-block warehouses have been discussed extensively in literature, and have proven to be interesting for both practical and theoretical purposes. The substantial literature mostly focuses on the expected order picking time and varies from exact optimization \citep{Theys2010} to simulation studies \citep{Roodbergen2001}. Extensions of the models, as well as queueing theoretical applications of the results, are also well-discussed. Notwithstanding the amount of research, the actual distribution of the order picking time was not yet determined. The design of warehouses, so far, has also been restricted to maximization of the throughput and minimization of the experienced delay. In this work, we extend the research by finding exact expressions for the order picking time distribution in single- and 2-block warehouses. These results can be used to design warehouses more carefully by also considering other performance statistics, e.g. tail probabilities. Additionally, distributional knowledge increases the accuracy of queueing theoretical analyses \citep{Gupta2010}.

This paper is built up as follows. In Section 2 we provide the model descriptions for the single- and 2-block warehouses, as well as a description of the return routing policy. In Section \ref{sec:Results} we derive the LST of the order picking time distribution for single- and 2-block warehouses, under return routing. We also show how to apply these results to two storage policies: random and class-based. In Section \ref{sec:NumRes} we use numerical inversion methods for Laplace Transforms to compare these storage policies as well as the single- and 2-block warehouses.

\section{Warehousing Model}

\label{sec:Model}
We consider single- and 2-block warehouses, consisting of $k$ equally spaced parallel aisles of length $l$ and width $w_a$. We assume that the warehouses consist of a single cross-aisle from which the pickers can enter the aisles. The difference between the single- and 2-block warehouses comes from the cross-aisle placement. In the single-block warehouse the cross-aisle is placed in the front, whereas in the 2-block warehouse, the cross-aisle is placed in the middle and splits the aisles in two sub-aisles (each of length $l/2$), see Figure \ref{fig:ReturnRouting}. In the sequel, we write sub-aisle $(i,j)$ for the sub-aisle in the $i$-th aisle and $j$-th block, that is: sub-aisle $(i,1)$ is the upper half of aisle $i$ in the 2-block warehouse. The depot at which the pickers receive and deliver their orders is located on the left-hand side of the cross-aisle, indicated as I/O in Figure \ref{fig:ReturnRouting}.

We assume that the pickers walk with a fixed speed $v$ and spend a random time $P$ to pick an arbitrary item, where $P$ has LST $\Phi_P(s)$. Furthermore, we assume that the time spent picking an item is independent of the pick locations and other picking times. The order size, $M$, is assumed to follow a Poisson distribution with mean $\lambda$.  For example, this is valid when item demands arrive according to a Poisson process and are batched as a single order after a fixed time.

We assume that the items in an order are independently located throughout the warehouse according to some storage policy. Let $p_{i,j}$ denote the probability that an arbitrary item in an order falls within sub-aisle $(i,j)$ and let $F_{i,j}(x)$ denote the probability that an arbitrary pick in this sub-aisle is within distance $x$ (as a fraction of the sub-aisle length) of the cross-aisle, for $i=1,...,k$ and $j=1,2$. Throughout this paper, we write $p_i$ and $F_i(\cdot)$ in the case of a single-block warehouse. This description allows for both continuous and discrete locations within aisles and represents a wide range of storage policies, e.g., the random and class-based storage policies. Under the former, the item locations in an order are uniformly distributed amongst and within the aisles. In the single-block warehouse this implies $p_i = 1/k$ and $F_i(x) = x$. In Section \ref{sec:Examples} we discuss the random and class-based policy more extensively. In particular, we show that the latter policy results in piece-wise linear $F_i(\cdot)$.

\begin{figure}[h]
    \centering
    \begin{subfigure}{0.45\textwidth}
        \includegraphics[width = \textwidth]{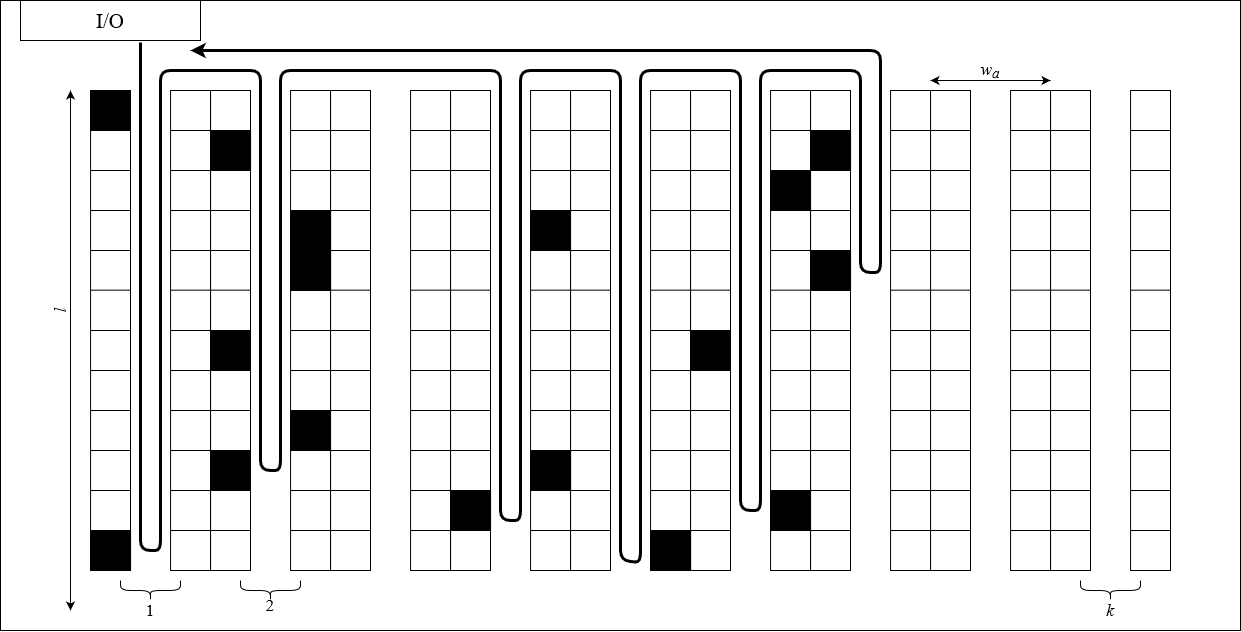}
        \caption{Single-block warehouse.}
        \label{fig:Return}
    \end{subfigure}\hspace{0.05\textwidth}
    \begin{subfigure}{0.45\textwidth}
        \includegraphics[width = \textwidth]{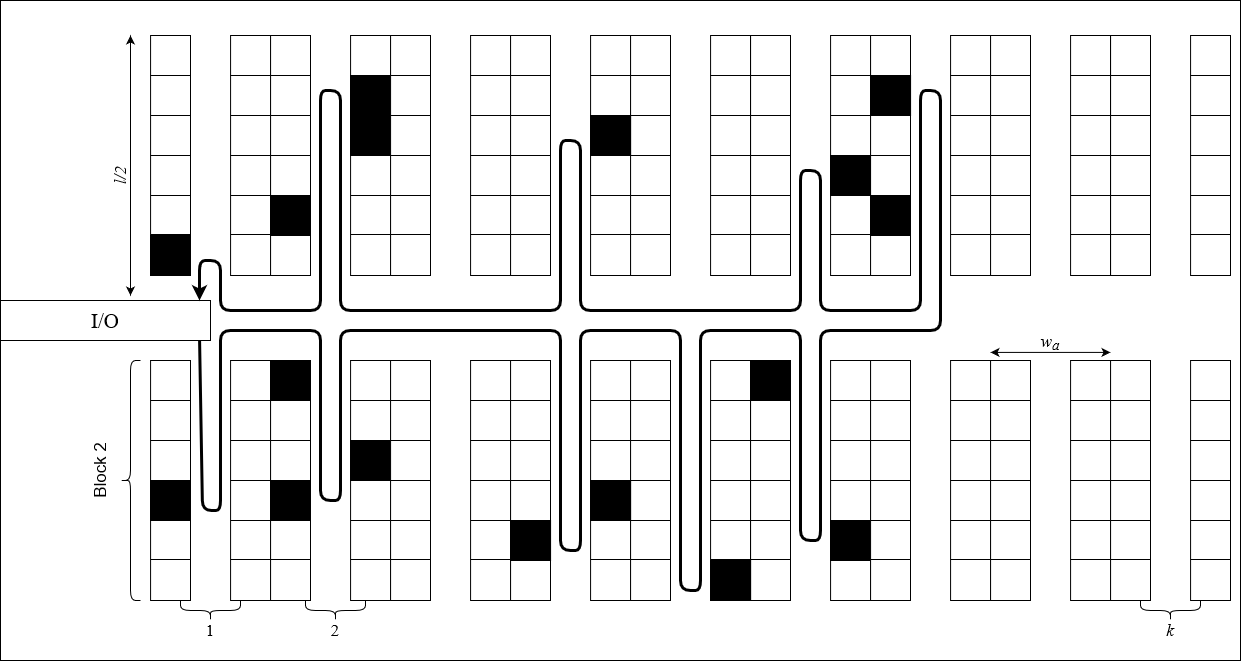}
        \caption{2-block warehouse.}
        \label{fig:Return_2block}
    \end{subfigure}
    \caption{Examples of the return routing policy in both single- and 2-block warehouses.}
    \label{fig:ReturnRouting}
\end{figure}

\subsection{Routing Policy}
In this paper we focus on the return routing policy, where pickers traverse each aisle until the furthest item that has to be picked and then traverse back. In the single-block warehouse this entails that the pickers walk across the front cross-aisle and enter each aisle in which they have to pick items. Within such aisles, the pickers pick all items up until the furthest item and then walk back to the cross-aisle. After all aisles are visited, the picker returns to the depot. An example of this policy is given in Figure \ref{fig:Return}.

We introduce the following notation. Let $N_i$ denote the number of items to pick in aisle $i$ and let $A_i$ be the location of the furthest item in aisle $i$ (as a fraction of the aisle length $l$). Furthermore, we write $k^+$ for the furthest aisle in which items have to be picked. Using this notation, we find that the picking time contribution of aisle $i$, denoted as $T_i$, is given by:
\begin{align}
\label{eq:Ret_Ti}
    T_i = \sum_{j=1}^{N_i} P_j + \frac{2l}{v}A_i + \frac{2w_a}{v}\mathbbm{1}\{i<k^+\}.
\end{align}
This relation consists of 3 different components: the time spent picking items, the traveling time within the aisle and the time spent traveling towards the next aisle. Since $k^+$ is the last aisle in which items have to be picked, the last term disappears when $i = k^+$. For $i>k^+$ the aisle will not have a picking time contribution and hence all terms will equal zero.

Consequently, the total order picking time, $T$, is given by:
\begin{align}
\label{eq:Ret_T_1}
    T = \begin{dcases}\sum_{i=1}^{k^+-1}T_i + T_{k^+} &\text{if } M>0,\\
         0 &\text{else.}
    \end{dcases}
\end{align}
Here, we explicitly distinguish between $T_i$ for $i<k^+$ and $T_{k^+}$ because of the third term in \eqref{eq:Ret_Ti} and the fact that aisle $k^+$ must consist of at least one pick, while the other aisles can still have zero picks.

Return routing in a 2-block warehouse is similar. The pickers walk across the cross-aisle and apply the return routing policy for a single-block warehouse on the lower block up until all items in said block are picked, hence they do not yet return to the I/O point. The pickers than traverse to the furthest (full) aisle with picks and apply the return routing policy on the upper block from right to left. When all items are picked, the pickers deliver their order at the depot-point. An illustration of this policy is given in Figure \ref{fig:Return_2block}.

In line with the notation for the single-block warehouse, we write $k^+$ for the furthest aisle in which items have to be picked, $N_{i,j}$ for the number of items to pick in sub-aisle $(i,j)$, and $A_{i,j}$ for the location of the furthest item in sub-aisle $(i,j)$ (as a fraction of the sub-aisle length $l/2$ and with respect to the cross-aisle). Then $T_i$ is given by:
\begin{align}
\label{eq:Ret2_Ti}
    T_i = \sum_{j=1}^{N_{i,1}+N_{i,2}} P_j + \frac{l}{v}\Big(A_{i,1}+A_{i,2}\Big) + \frac{2w_a}{v} \mathbbm{1}\{i< k^+\}.
\end{align}
Remark that relation \eqref{eq:Ret_T_1} still holds for the 2-block warehouse.

\section{Results}
\label{sec:Results}

In this section we derive the LST of the order picking time distribution in both single- (Section \ref{sec:SingleBlock}) and 2-block warehouses (Section \ref{sec:2Block}). Afterwards, in Section \ref{sec:Examples} we illustrate how these results can be applied to two well-known storage policies. The analysis in this section is based upon the key assumption that the total order size follows a Poisson distribution. Since the number of picks in the different (sub-)aisles follows a multinomial distribution with the order size as number of trials, we have that the numbers of picks in the different aisles are independent under this assumption and follow a Poisson distribution with parameter $p_{i,j}\lambda$.

\subsection{Single-block warehouse}
\label{sec:SingleBlock}
Recall the relation for the order picking time in Equation \eqref{eq:Ret_T_1}, which splits the order picking time in the picking time contribution of the last aisle and the sum of the contributions of the previous aisles. The latter term is a random sum of ($k^+-1$) i.i.d. random variables ($T_i$), hence deriving the distribution of the order picking time requires to derive the distribution of both $k^+$ and $T_i$. We start by remarking that $k^+$ is the first aisle with picks coming from aisle $k$ to aisle $1$. For the aisle contributions $T_i$, we notice that $A_i$ is the furthest pick location in the aisle, i.e., $A_i$ is the maximum of $N_i$ i.i.d. random variables. Therefore we have that: $\mathbb{P}(A_i\leq x\vert N_i = n) = F_i(x)^n$. Using these insights we find the following:

\begin{lemma}
\label{lemma:Ret_mainLemma}
The following equations hold:
\begin{align}
    \label{eq:kplusPGF}
    &\mathbb{P}(k^+ = j) = \exp\Big(-\lambda \sum_{i=j+1}^k p_i\Big)\big[1-\exp(-\lambda p_j)\big] \quad \text{for } j = 1,...,k,\\
    \label{eq:Ret_TiLST}
    &\mathbb{E}[\exp(-sT_{i})\vert k^+>i] = \exp\Big(-\lambda p_i-\frac{2w_as}{v}\Big)\bigg[
    \begin{aligned}[t]
        &\exp\Big(\lambda p_i \Phi_P(s)-\frac{2ls}{v}\Big)\\ &+\frac{2ls}{v}\int_{x=0}^1 \exp\Big(-\frac{2ls}{v}x+\lambda p_i\Phi_P(s)F_i(x)\Big)\dx x\bigg],
    \end{aligned}
\intertext{and}
    \label{eq:Ret_Tk}
     &\mathbb{E}[\exp(-sT_{i})\vert k^+ = i] = \frac{\exp(2w_as/v)\mathbb{E}[\exp(-sT_i)\vert k^+ > i]-\exp(-\lambda p_i)}{1-\exp(-\lambda p_i)}.
\end{align}
\end{lemma}
\begin{proof}
The first statement follows immediately from the given that $k^+$ is the right-most non-empty aisle in the warehouse, combined with the fact that $\mathbb{P}(N_i = 0)=\exp(-\lambda p_i)$.

For \eqref{eq:Ret_TiLST} we distinguish between the cases $N_i = 0$ and $N_i > 0$:
\begin{align}
\label{eq:Ret_mainLemma1}
    \mathbb{E}[\exp(-sT_i)\vert k^+>i] &= \exp\Big(-\lambda p_i-\frac{2w_as}{v}\Big) +  \mathbb{E}[\exp(-sT_i)\mathbbm{1}\{N_i>0\}\vert k^+>i].
\end{align}
For the second term we realize that, conditional on $N_i$, we have by partial integration that:
\[\mathbb{E}[\exp(-sA_i)\vert N_i = n] = \exp(-s) + s\int_{x=0}^1 \mathbb{P}(A_i \leq x\vert N_i = n) \exp(-sx)\dx x.\]
Now using $\mathbb{P}(A_i\leq x\vert N_i = n) = F_i(x)^n$ and by summing over all possible values of $N_i$ we find:
\begin{align*}
    \mathbb{E}[\exp(-sT_i)\mathbbm{1}\{N_i>0\}\vert k^+>i]&= \exp\Big(-\lambda p_i-\frac{2w_as}{v}\Big)\sum_{n=1}^\infty \frac{\lambda^np_i^n}{n!}\Phi_P(s)^n\bigg[\exp\Big(-\frac{2ls}{v}\Big)+\frac{2ls}{v}\int_{x=0}^1 \exp(-\frac{2ls}{v}x)F_i(x)^n\dx x\bigg]\\
    &=\exp\Big(-\lambda p_i-\frac{2w_as}{v}\Big)\bigg[\begin{aligned}[t]
    &\exp\Big(\lambda p_i\Phi_P(s)-\frac{2ls}{v}\Big)-\exp\Big(-\frac{2ls}{v}\Big)\\
    &+\frac{2ls}{v}\int_{x=0}^1 \bigg\{\exp\Big(-\frac{2ls}{v}x+\lambda p_i\Phi_P(s)F_i(x)\Big)-\exp\Big(-\frac{2ls}{v}x\Big)\bigg\}\dx x\bigg]
    \end{aligned}\\
    &=\exp\Big(-\lambda p_i-\frac{2w_as}{v}\Big)\bigg[\begin{aligned}[t]
    &\exp\Big(\lambda p_i\Phi_P(s)-\frac{2ls}{v}\Big)-\exp\Big(-\frac{2ls}{v}\Big)\\
    &+\frac{2ls}{v}\int_{x=0}^1 \exp\Big(-\frac{2ls}{v}x+\lambda p_i\Phi_P(s)F_i(x)\Big)\dx x + \exp\Big(-\frac{2ls}{v}\Big)-1\bigg].
    \end{aligned}
\end{align*}
Remark that the terms $\exp(-2ls/v)$ cancel and that, after substitution in \eqref{eq:Ret_mainLemma1}, the $-1$ term cancels against the term for $N_i=0$, proving \eqref{eq:Ret_TiLST}.

Lastly, for \eqref{eq:Ret_Tk} we note that:
\begin{align*}
    \Big\{T_{i}\Big\vert k^+ =i\Big\} \,&\overset{d}{=} \bigg\{\sum_{m=1}^{N_i}P_m\bigg\vert N_i>0\bigg\} + \frac{2l}{v}\Big\{A_i\Big\vert N_i>0\Big\} \overset{d}{=} \frac{-2w_a}{v}+\Big\{T_i\Big\vert k^+ > i,N_i > 0\Big\}.
\end{align*}
For the LST of $\Big\{T_{i}\Big\vert k^+ =i\Big\}$ we can thus use the distribution on the right-hand side. By then applying the law of total probability on the event $N_i > 0$ we find:
\begin{align*}
     \mathbb{E}\Big[\exp\Big(-sT_i+\frac{2w_as}{v}\Big)\Big\vert k^+>i,N_i>0\Big]
    &=\frac{\exp(2w_as/v)\mathbb{E}[\exp(-sT_i\Big)\Big\vert k^+>i]-\exp(-\lambda p_i)}{1-\exp(-\lambda p_i)}.\qedhere
\end{align*}
\end{proof}

The results in Lemma \ref{lemma:Ret_mainLemma} provide the main ingredients for the LST of the total order picking time. Starting from \eqref{eq:Ret_T_1}, we sum over the possible values of $k^+$. Since the aisle contributions, conditional on $k^+$, are independent (up to aisle $k^+$) we have that the LST is simply the product of the transforms of each aisle contribution. In conclusion we get:

    \begin{theorem}
    \label{thm:Return}
    The total order picking time in a single-block warehouse, under return routing, has the following Laplace-Stieltjes transform:
    \begin{align}
    \label{eq:Return_LST_T}
        \mathbb{E}[\exp(-sT)] = \exp(-\lambda) + \sum_{i=1}^{k} \exp\Big(-\lambda \sum_{m=i+1}^k p_m\Big)\Big[\exp\Big(\frac{2w_as}{v}\Big)\psi_{T_i}(s)-\exp\big(-\lambda p_i\big)\Big]\prod_{l=1}^{i-1} \psi_{T_l}(s),
    \end{align}
    where $ \psi_{T_l}(s) = \mathbb{E}[\exp(-sT_l)\vert k^+>l]$ for $l = 1,2,...k$, given in Lemma \ref{lemma:Ret_mainLemma}.
    \end{theorem}


\subsection{2-block warehouse}
\label{sec:2Block}
The order picking time, $T$, in the 2-block warehouse still satisfies the relation in \eqref{eq:Ret_T_1}, where $T_i$ is given in \eqref{eq:Ret2_Ti}. The main idea behind Theorem \ref{thm:Return} therefore still applies, resulting in the same relation between the LST of the total order picking time and the aisle contributions. The main difference between the transforms comes from the contribution of each aisle, which we analyze by splitting this contribution into two independent parts: one for each block. This results in Theorem \ref{thm:Return2}.
    \begin{theorem}
    \label{thm:Return2}
   The total order picking time in a 2-block warehouse, under return routing, has the following Laplace-Stieltjes transform:
    \begin{align}
        \mathbb{E}[\exp(-sT)] = \exp(-\lambda) + \sum_{i=1}^{k} \Big[&\exp\Big(\frac{2w_as}{v}\Big)\psi_{T_{i,1}}(s)\psi_{T_{i,2}}(s)-\exp\big(-\lambda (p_{i,1}+p_{i,2})\big)\Big]\\
        &\times\exp\Big(-\lambda \sum_{m=i+1}^k (p_{m,1}+p_{m,2})\Big)\prod_{l=1}^{i-1} \psi_{T_{l,1}}(s)\psi_{T_{l,2}}(s),\nonumber
    \end{align}
    where:
    \begin{align}
    \label{eq:Return2_Ti}
        \psi_{T_{l,j}}(s) = \exp\Big(-\lambda p_{l,j}-\frac{w_as}{v}\Big)\Big[\exp\Big(\lambda p_{l,j} \Phi_P(s)-\frac{ls}{v}\Big) + \frac{ls}{v}\int_{x=0}^1 \exp\Big(-\frac{ls}{v}x+\lambda p_{l,j}\Phi_P(s)F_{l,j}(x)\Big)\dx x\Big],
    \end{align}
    for $l = 1,2,...k$ and $j=1,2$.
    \end{theorem}
    \begin{proof}
By similar reasoning as in Lemma \ref{lemma:Ret_mainLemma} we have:
\begin{align}
    \mathbb{P}(k^+ = j) = \exp\Big(-\lambda \sum_{i=j+1}^k (p_{i,1}+p_{i,2})\Big)\big[1-\exp\big(-\lambda (p_{j,1}+p_{j,2})\big)\big] \quad \text{for: } j = 1,...,k.
\end{align}
For the aisle contributions we note that $T_i$ can be written as the sum of the contributions in the upper and lower block and also split the horizontal travel distance contribution in two, i.e.:
\begin{align}
    T_i = T_{i,1}+T_{i,2} =
    \sum_{j=1}^{N_{i,1}} P_j + \frac{l}{v}A_{i,1}+\frac{w_a}{v}\mathbbm{1}\{i < k^+\} + \sum_{j=1}^{N_{i,2}} P_j +  \frac{l}{v}A_{i,2} + \frac{w_a}{v} \mathbbm{1}\{i < k^+\}.
\end{align}
Observe that $T_{i,1},T_{i,2}$ are similar to $T_i$ in the single-block warehouse. In fact, they have the same distribution for $i<k^+$ when we take: $w_a/2,\lambda p_{i,j}$ and $l/2$ instead of $w_a,\lambda p_i$ and $l$. We can thus use the result in Lemma \ref{lemma:Ret_mainLemma} for the contribution of a single sub-aisle, resulting in \eqref{eq:Return2_Ti}. Again, these contributions are independent for $i<k^+$ and hence the transform of $T_{i,1}+T_{i,2}$ is the product of the transforms: $\psi_{T_i}(s) = \psi_{T_{i,1}}(s)\psi_{T_{i,2}}(s)$. For $i=k^+$ we, however, have to be careful, since either $N_{i,1}$ or $N_{i,2}$ can still be zero. For this aisle, we remark:
\begin{align*}
    \mathbb{E}\big[\exp(-sT_i)\big\vert i =k^+\big]
    &=\mathbb{E}\bigg[\exp\Big(-s\Big[-2\frac{w_a}{v}+T_{i,1}+T_{i,2}\Big]\Big)\bigg\vert i <k^+, N_i>0\bigg].
\end{align*}
Similar to \eqref{eq:Ret_Tk} we apply the law of total probability. By independence of the $T_{i,1}, T_{i,2}$ we now find:
\begin{align*}
     \mathbb{E}\big[\exp(-sT_i)\vert i =k^+\big] &= \frac{\exp(2w_as/v)\psi_{T_{i,1}}(s)\psi_{T_{i,2}}(s)-\exp(-\lambda(p_{i,1}+p_{i,2}))}{1-\exp(-\lambda(p_{i,1}+p_{i,2}))}.
\end{align*}
The denominator cancels against the term in $\mathbb{P}(k^+ = i)$ and we therefore find the result as in the theorem.
\end{proof}

\subsection{Storage policy examples}
\label{sec:Examples}
We apply these results for two different storage policies.
Firstly, we discuss the random storage policy, under which each item in an order is uniformly located amongst and within the aisles. Under this policy, we therefore have $p_i= 1/k$ and $F_i(x) = x$ for all $i$. Substituting this in Theorem \ref{thm:Return} results in the following expression:
\begin{lemma}
Under the random storage policy the LST of the total order picking time, for the single-block warehouse, is given by:
\begin{equation}
\label{eq:Return_Random}
\begin{aligned}
    \mathbb{E}[\exp(-sT)] = \exp(-\lambda)\bigg\{1+&\exp\Big(\frac{\lambda}{k}\Big)\Big[\exp\Big(\frac{2w_as}{v}\Big)\psi_{T_i}(s)-\exp\Big(-\frac{\lambda}{k}\Big)\Big] \frac{1-\exp(\lambda)\psi_{T_i}(s)^k}{1-\exp(\lambda/k)\psi_{T_i}(s)}\bigg\},
\end{aligned}
\end{equation}
where:
\begin{align}
\label{eq:Return_Random_Ti}
\psi_{T_i}(s) = \begin{aligned}[t]\exp&\Big(-\frac{\lambda}{k}-\frac{2w_as}{v}\Big)\cdot\bigg\{1+\frac{\Phi_P(s)\lambda v}{\Phi_P(s)\lambda v-2lks}\bigg[\exp\Big( \frac{\lambda}{k}\Phi_P(s)-\frac{2ls}{v}\Big) - 1\bigg]\bigg\}.
\end{aligned}
\end{align}
\end{lemma}
\begin{proof}
Using $p_i = 1/k, F_i(x) = x$ in Theorem \ref{thm:Return} gives:
\begin{align*}
    \psi_{T_i}(s) &=\exp\Big(-\frac{\lambda}{k}-\frac{2w_as}{v}\Big)\bigg[
        \exp\Big(\frac{\lambda}{k} \Phi_P(s)-\frac{2ls}{v}\Big)+\frac{2ls}{v}\int_{x=0}^1 \exp\Big(-\frac{2ls}{v}x+\frac{\lambda}{k}\Phi_P(s)x\Big)\dx x\bigg]\\
        &=\exp\Big(-\frac{\lambda}{k}-\frac{2w_as}{v}\Big)\bigg[
        \exp\Big(\frac{\lambda}{k} \Phi_P(s)-\frac{2ls}{v}\Big)+\frac{2lks}{\Phi_P(s)\lambda v-2lks}\Big[\exp\Big(\frac{\lambda}{k} \Phi_P(s)-\frac{2ls}{v}\Big)-1\Big]\bigg],
\end{align*}
which results in \eqref{eq:Return_Random_Ti} by adding the terms. Next turning to \eqref{eq:Return_LST_T} we use that $\psi_{T_i}(\cdot)$ is the same for all $i$ and hence:
\begin{align}
    \mathbb{E}[\exp(-sT)] &= \exp(-\lambda) + \sum_{j=1}^{k} \Big[\exp\Big(\frac{2w_as}{v}\Big)\psi_{T_i}(s)-\exp\Big(-\frac{\lambda}{k}\Big)\Big]\exp\Big(-\frac{\lambda (k-j)}{k}\Big)\psi_{T_i}(s)^{j-1}\nonumber\\
    \label{eq:Return_Random_eq}
    &=\exp(-\lambda) + \Big[\exp\Big(\frac{2w_as}{v}\Big)\psi_{T_i}(s)-\exp\Big(-\frac{\lambda}{k}\Big)\Big]\exp\Big(-\frac{\lambda (k-1)}{k}\Big)\sum_{j=0}^{k-1}\Big[\exp\Big(\frac{\lambda}{k}\Big)\psi_{T_i}(s)\Big]^{j}.
\end{align}
Rewriting the geometric sum in \eqref{eq:Return_Random_eq} yields \eqref{eq:Return_Random}.
\end{proof}

\textbf{Note:} From this equation one can find the first moment of $T$ by evaluating the derivative of the LST at $0$. The resulting expectation coincides with the results in \citet{Engels2022}:
\begin{align*}
     \mathbb{E}[T] &= \frac{2lk^2}{v\lambda}\Big[1-\exp\Big(-\frac{\lambda}{k}\Big)\Big] + \frac{2w_a}{v}\Big[k-\frac{1-\exp(-\lambda)}{1-\exp(-\lambda/k)}\Big].
\end{align*}

In the 2-block warehouse the results are very similar: relation \eqref{eq:Return_Random} still holds, yet now we have $p_{i,j} = 1/(2k)$. Consequently, $\psi_{T_i}(\cdot)$ is given by the square of the expression in \eqref{eq:Return_Random_Ti} with $k$ replaced by $2k$ and $w_a$ replaced by $w_a/2$. The square comes from the fact that $\psi_{T_{i,1}}(\cdot)$ and $\psi_{T_{i,2}}(\cdot)$ are equal.\\

In practice, warehouses often exploit knowledge on the demand of items. For example, storing high-demand items in close locations results in lower order picking times. One well-known example of such a storage policy is the class-based policy, see e.g. \citep{Dijkstra2017}.

Under the class based storage policy, each item is grouped in a certain demand class. Each of these classes has a separate area within the warehouse in which the corresponding items are located uniformly at random. Under this policy, one can thus prevent long routes by placing high demand items (high class) close to the cross-aisles.

Below we provide a mathematical definition for class-based storage.
We follow the notation and definition in \citet{Dijkstra2017} and assume that there are a total of $Q$ classes and write $\tilde{p}_q$ for the probability that an arbitrary item in an order is of class $q$. We consider the case where each aisle consists of $Q$ different ranges, one for each class. Let $u_q^{(i)}$ be the upper-bound of the locations of type $q$ items in aisle $i$, such that all class $q$ items in aisle $i$ fall between $u_{q-1}^{(i)}$ and $u_{q}^{(i)}$. Here we use the notion that $u_0^{(i)} = 0$ and $u_Q^{(i)}=1$. We thus have that type $q$ items make up a total of $f_q$ space, where $f_q$ is given by: $f_q:=\sum_{i=1}^k \big(u_q^{(i)} - u_{q-1}^{(i)}\big)$.

Although the definition above is a special case of the class-based storage, it is known to provide optimal solutions, w.r.t. the average route-length, for return routing when $\tilde{p}_1>\tilde{p}_2>...>\tilde{p}_Q$ \citep{Dijkstra2017}.

In case of 2 classes we can translate this to our model by noting:
\begin{align*}
    p_i &=  \sum_{q=1}^2 \mathbb{P}(\text{item of class }q)\mathbb{P}(\text{item in aisle }i\vert \text{item of class } q)=\tilde{p}_1\cdot \frac{u_1^{(i)}}{f_1} + \tilde{p}_2\cdot \frac{1-u_1^{(i)}}{f_2},\\
    F_i(x) &=\sum_{q=1}^2\mathbb{P}(\text{item of class }q\vert \text{item in aisle }i)\mathbb{P}(\text{Location in aisle $i$ before } x\vert \text{item of class }q)\\
    &=\begin{dcases}
    \frac{\tilde{p}_1}{p_i}\frac{u_1^{(i)}}{f_1}\cdot \frac{x}{u_1^{(i)}} = \frac{\tilde{p}_1}{p_i}\frac{x}{f_1} &\text{if } x< u_1^{(i)},\\
    \frac{\tilde{p}_1}{p_i}\frac{u_1^{(i)}}{f_1} + \frac{\tilde{p}_2}{p_i}\frac{x-u_1^{(i)}}{f_2} & \text{if } x \geq u_1^{(i)}.
    \end{dcases}
\end{align*}
One can substitute these relations into Theorem \ref{thm:Return} yielding the LST of the total order picking time. In fact, one can analytically find the integral in \eqref{eq:Ret_TiLST} by splitting the integral in two parts, one for each class. The case for a general number of classes follows similarly.\\

\textbf{Remark:} This class-based storage model is a continuous version of the model in \citet{Dijkstra2017}. By letting the number of locations in \citet{Dijkstra2017} go to infinity, and appropriately scaling the probabilities, the model becomes equivalent to the one described here.

\section{Numerical Illustration}
\label{sec:NumRes}
In this section we focus on the comparison of routing and storage policies. We do this by numeric inversion of LST's using the method of Abate and Whitt \citep{Abate1992}.

We consider a warehouse model with $l=20m,w_a = 2.5m, v= 0.83m/s$ and let $P\sim\mathrm{Exp}(1/5)$. For the class-based storage policy, we consider the example in \citet{Dijkstra2017} for a $50/30/20$ distribution of classes and a $20/30/50\%$ of space required per class, where $k = 15$ and $\lambda = 10$. We take the class-based areas as graphically given in \citet{Dijkstra2017}. For the 2-block warehouse we assign the same area division to the lower block and mirror the areas in the upper block, see the illustration in Figure \ref{fig:StorageIllustration}; the various colours indicate the regions in which the items of the corresponding class are stored. Besides the class-based storage, we also consider random storage. The picking time densities for these storage policies are shown in Figure \ref{fig:Storage}.\\
\begin{figure}[h]
    \centering
    \begin{subfigure}{0.45\textwidth}
        \includegraphics[width = \textwidth]{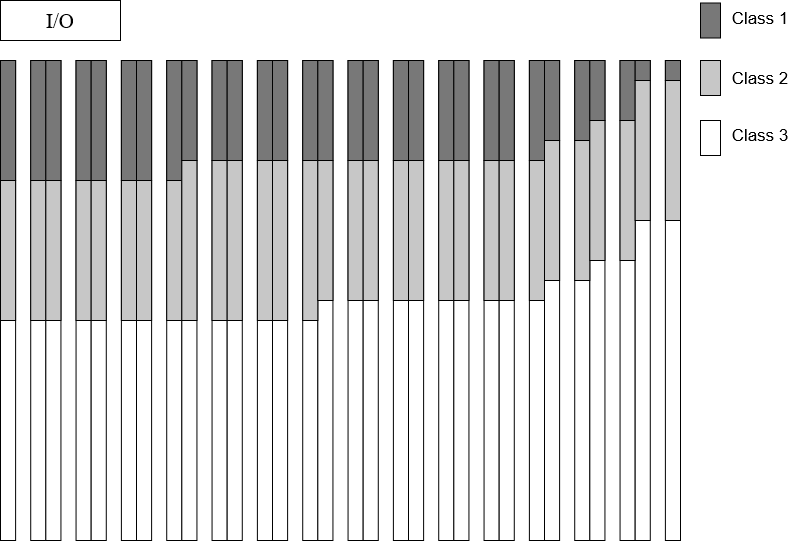}
        \caption{Single-block warehouse.}
    \end{subfigure}\hspace{0.05\textwidth}
    \begin{subfigure}{0.45\textwidth}
        \includegraphics[width = \textwidth]{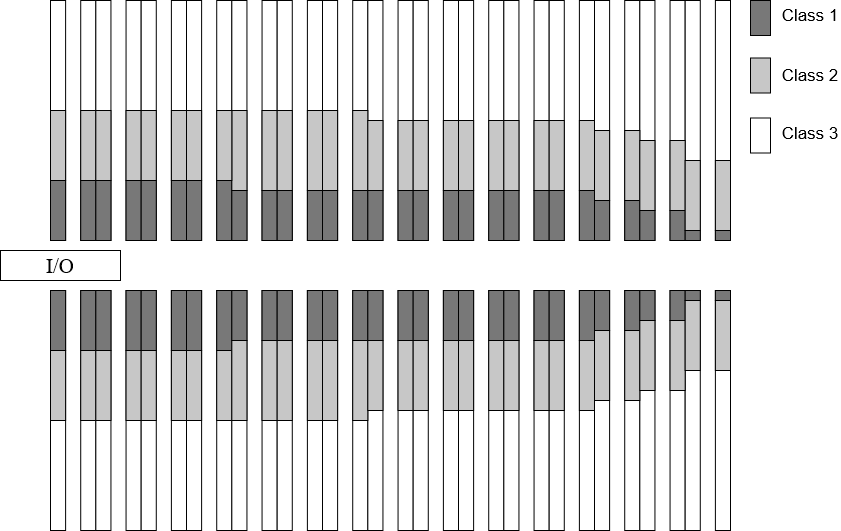}
        \caption{2-block warehouse.}
    \end{subfigure}
    \caption{Illustrations of the class-based storage areas for the numerical example.}
    \label{fig:StorageIllustration}
\end{figure}

\begin{figure}[h]
    \centering
    \includegraphics[width = 10cm]{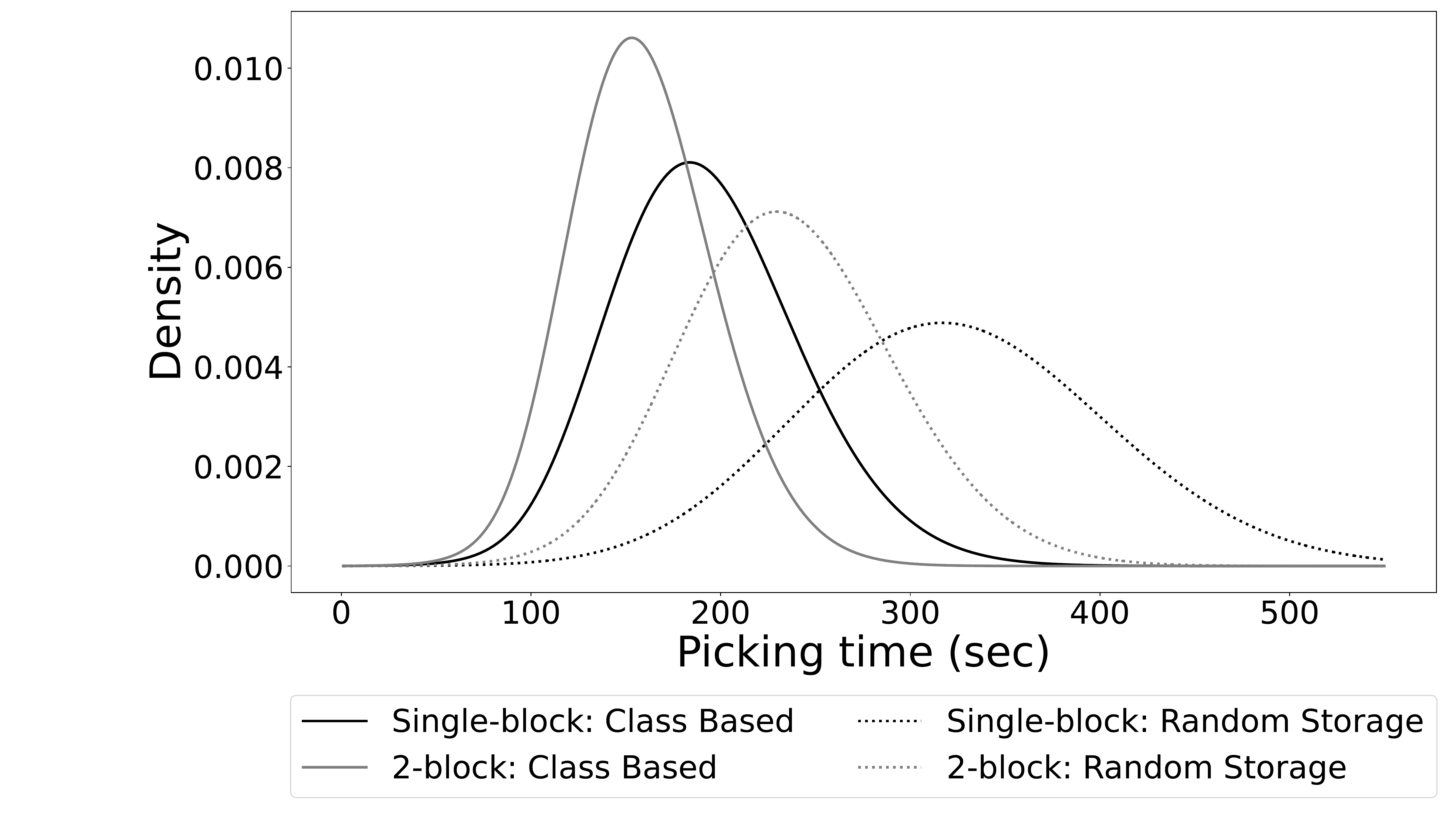}
    \caption{Illustration of the density of the total order picking time under the Class-Based Storage and Random Storage policies.}
    \label{fig:Storage}
\end{figure}

We can see that the single-block warehouse is outperformed by the multi-block warehouse and the class-based storage, in turn, outperforms the random storage policy. Similar comparisons have previously been found in literature, yet the quantification that is allowed by knowing the complete distribution is new. For example the plots reveal a similarity to the normal density. This similarity can be explained by relation \eqref{eq:Ret_T_1}, showing that the order total picking time is a random sum of independent random variables. By the Central Limit Theorem for random sums of independent random variables we thus know that the total order picking time distribution will converge to a normal distribution when both $k$ and $\lambda$ go to infinity at the same rate \citep{Shantikumar}.\\

\section{Conclusion}
In this paper we have derived exact formulas for the Laplace-Stieltjes Transform of the total order picking time under return routing for a number of the most important storage policies, of which we discussed two in detail. The formulas in this paper allow one to exactly determine the order picking time distribution and can therefore be used to gain a deeper insight in the order picking problem, as is highlighted in the numerical results.

The results in this paper exploit the independence between the aisles, which is a result of the assumption that the order sizes follows a Poisson distribution. While restrictive, these results can also be used to approximate the order picking time distribution when the order size follows a mixed Poisson distribution, by conditioning over the mixing parameter and approximating the integral. This includes the case of a negative binomial distribution (mixed Poisson in which the Poisson rate is Gamma distributed) or the case where the rate of the Poisson process (for the demand) changes over time.

This paper restricts the derivation to return routing in single- and 2-block warehouses, but forms a foundation for the derivation of the order picking time LST for multi-block warehouses or other routing policies. In particular, the results can be generalized for midpoint, largest-gap and S-shaped routing in single-block warehouses deploying the random storage policy. For more advanced storage policies, the formulas quickly tend to become computationally involved. Furthermore, warehouses with multiple blocks can be considered. Yet, efficient routing policy definitions in multi-block warehouses are scarce and often involve many steps and cases to take into account \citep[see e.g.][]{Roodbergen2008}.

Finally, we emphasize that the results of this paper can be used for a large range of performance statistics. E.g., they can be used to obtain good bounds on the tail probabilities of the order picking time, as well as approximations of the order-lead time distribution, viewing the order picking time as service time in a multi-server queue with $c$ pickers and applying a result of \citet{VanHoorn1982}. As a consequence, this paper allows for a more direct analysis of statistics in warehouses, which enables a better comparison of warehouse storage policies.

\section*{Acknowledgements}
The research of Tim Engels and Onno Boxma is partly funded by the NWO Gravitation project NETWORKS, grant number 024.002.003.



\bibliography{OR_letters_revisit.bib}

\end{document}